\documentclass[10pt]{article}         
\usepackage{amsmath,amsfonts,amssymb,mathrsfs,amsthm,tikz-cd,tikz,array,quiver,url}
\usetikzlibrary{shapes.geometric, arrows}
\tikzstyle{startstop} = [rectangle, rounded corners, minimum width=3cm, minimum height=1cm,text centered, draw=black, fill=red!0]
\tikzstyle{io} = [trapezium, trapezium left angle=70, trapezium right angle=110, minimum width=3cm, minimum height=1cm, text centered, draw=black, fill=blue!0]
\tikzstyle{process} = [rectangle, minimum width=3cm, minimum height=1cm, text centered, draw=black, fill=orange!0]
\tikzstyle{decision} = [diamond, minimum width=3cm, minimum height=1cm, text centered, draw=black, fill=green!0]
\tikzstyle{arrow} = [thick,->,>=stealth]  
\usepackage[utf8]{inputenc}
\usepackage{hyperref}
\usepackage{cite}

\newtheorem{thm}{Theorem}[section]

\theoremstyle{definition}
\newtheorem{defn}[thm]{Definition}
\newtheorem{prop}[thm]{Proposition}
\newtheorem{lem}[thm]{Lemma}

\theoremstyle{remark}

\title{Open Homomorphisms between $m$-step Solvable Galois Groups Compatible with the Cyclotomic Characters}
\author{
  Yu Mao\\
  \texttt{ym382@exeter.ac.uk}
  \and
  Mohamed Sa\"idi\\
  \texttt{m.saidi@exeter.ac.uk}
}

\date{}
\begin{document}
\maketitle
\begin{abstract}
In \cite{Ho3}, Hoshi proved that open homomorphisms between solvably closed Galois groups of number fields which are compatible with the cyclotomic characters arise from field embeddings. In this paper, we will prove an $m$-step solvable version of Hoshi's result. More precisely, if $K$ and $L$ are number fields, we will prove that given an open homomorphism between the maximal $m+3$-step solvable Galois groups of $K$ and $L$, where $m \geq 2$, and the induced open homomorphism between the corresponding maximal $m$-step solvable Galois groups, then the latter arises from a field embedding if and only if the open homomorphism between the $m+3$-step solvable (and hence also the $m$-step solvable) Galois groups is compatible with the cyclotomic characters of $K$ and $L$.
\end{abstract}
\newpage
\section*{Conventions}
Throughout this paper, we use the following conventions.\par
\begin{itemize}
\item Let $K$ be a number field, we write $\mathscr{P}_K^{\text{fin}}$ for the set of non-archimedean primes of $K$. 
\item Let $\mathfrak{p} \in \mathscr{P}_K^{\text{fin}}$, we write $p_{\mathfrak{p}}, d_{\mathfrak{p}},f_{\mathfrak{p}},e_{\mathfrak{p}}$ for the residue characteristic of $\mathfrak{p}$, the degree $[K_{\mathfrak{p}}: \mathbb{Q}_{p_{\mathfrak{p}}}]$, where $K_{\mathfrak{p}}$ is the completion of $K$ at $\mathfrak{p}$, the inertia degree of $\mathfrak{p}$ and the ramification index of $\mathfrak{p}$ respectively.
\item We write $\mathscr{P}_K^{\text{fin},f=1}$ for the subset of $\mathscr{P}_K^{\text{fin}}$ consisting of those primes $\mathfrak{p}$ with $f_{\mathfrak{p}} = 1$. 
\item Let $G$ be a profinite group and let $i \geq 0$ be an integer. We write $G^{[i]}$ for the (closure of the) $i$-th derived subgroup of $G$. Thus,
$$
G^{[0]} := G ~;~ G^{[i]} := \overline{[G^{[i-1]},G^{[i-1]}]}
$$
for all $i \geq 1$.
\item Let $G$ be a profinite group and let $i \geq 0$ be an integer. We write $G^i$ for the maximal $i$-step solvable quotient of $G$, which is defined to be $G/G^{[i]}$. In particular, $G^1$ coincides with the abelianisation of $G$.
\item Let $G$ be a profinite group and let $j \geq i \geq 0$ be integers. We write $G^{[j,i]} := \text{ker}(G^j \twoheadrightarrow G^i)$.

\item Let $G$ be a profinite group. We write $G^{\text{sol}}$ for the maximal prosolvable quotient of $G$. We have
$$
G^{\text{sol}} := \varprojlim_{ n \geq 0} ~G^n.
$$
\item Let $K$ be a number field and fix a separable closure $K^{\text{sep}}$ of $K$. We write $K^{\text{sol}} := (K^{\text{sep}})^{\text{ker}(G_K \twoheadrightarrow G_K^{\text{sol}})}$ for the maximal prosolvable extension of $K$. Let $m \geq 1$, we write $K_m := (K^{\text{sep}})^{G_K^{[m]}}$ for the maximal $m$-step solvable extension of $K$. In particular, $K_1$ coincides with the maximal abelian extension of $K$. 
\item All homomorphisms between profinite groups are assumed to be continuous.
\end{itemize}
\section{Introduction}
Recall that the Neukirch-Uchida Theorem asserts that two number fields $K$ and $L$ are isomorphic if and only if $G_K^{\text{sol}}$ and $G_L^{\text{sol}}$ are isomorphic as profinite groups (c.f. \cite{Uchida1}). It is natural to ask what happens if we replace an isomorphism between $G_K^{\text{sol}}$ and $G_L^{\text{sol}}$ by an open homomorphism? \par 
In 1981, Uchida gave a conditional solution to this question:
\begin{thm}[Theorem 2 in \cite{Uchida3}]
Let $K$ and $L$ be number fields, and let
$$
\sigma: G_L^{\text{sol}} \to G_K^{\text{sol}}
$$
be a continuous homomorphism such that the following condition holds. For each non-archimedean prime $\tilde{\mathfrak{p}}$ of $L^{\text{sol}}$, there exists a non-archimedean prime $\tilde{\mathfrak{q}}$ of $K^{\text{sol}}$ such that $\sigma$ restricts to an open homomorphism $\sigma: D_{\tilde{\mathfrak{p}}} \to D_{\tilde{\mathfrak{q}}}$ where $D_{\tilde{\mathfrak{p}}}$ and $D_{\tilde{\mathfrak{q}}}$ are decomposition groups at $\tilde{\mathfrak{p}}$ and $\tilde{\mathfrak{q}}$, respectively. Then there exists a unique field embedding $\tau:K^{\text{sol}} \hookrightarrow L^{\text{sol}}$ such that $\tau \sigma(g) = g \tau$ for all $g \in G_{L}^{\text{sol}}$. Furthermore, $\tau$ restricts to a field embedding $K \hookrightarrow L$.
\end{thm}
Later in 2022, Hoshi proved the following variant of Theorem 1.1:
\begin{thm}[Theorem 3.4 in \cite{Ho3}]
Let $K$ and $L$ be number fields, and let $\alpha: G_L^{\text{sol}} \to G_K^{\text{sol}}$ be an open homomorphism of profinite groups. Then the followings are equivalent. \par 
(i) $\alpha$ arises uniquely from a field embedding. That is, there exists a unique field embedding $\beta: K^{\text{sol}} \to L^{\text{sol}}$ such that $\beta \alpha(g) = g \beta$ for all $g \in G_L^{\text{sol}}$. \par 
(ii) The following diagram commutes:
$$
\begin{tikzcd}
	{G_L^{\text{sol}}} && {G_K^{\text{sol}}} \\
	& {\widehat{\mathbb{Z}}^{\times}}
	\arrow[from=1-1, to=1-3]
	\arrow["{\chi_L}", from=1-1, to=2-2]
	\arrow["{\chi_K}"', from=1-3, to=2-2]
\end{tikzcd}
$$
where $\chi_L$ and $\chi_K$ are the cyclotomic characters associated to $L$ and $K$, respectively.
\end{thm}
On the other hand, in 2019, Sa\"idi and Tamagawa proved an $m$-step solvable version of the Neukirch-Uchida Theorem:
\begin{thm}[Theorem 2 in \cite{ST1}]
Let $K,L$ be number fields and let $\sigma_{m+3}: G_L^{m+3} \xrightarrow{\sim} G_K^{m+3}$ be an isomorphism, which induces an isomorphism $\sigma_m: G_L^m \xrightarrow{\sim} G_K^m$. Then the following assertions hold.\par 
(i) If $m \geq 0$, there exists a field isomorphism $\tau_m : K_m \xrightarrow{\sim} L_m$ such that $\sigma_m(g) = \tau_m^{-1} g \tau_m$ for all $g \in G_L^m$, and $\tau_m$ restricts to a field isomorphism $K \xrightarrow{\sim} L$. \par 
(ii) If $m \geq 2$ (resp. $m=1$), then the isomorphism $\tau_m : K_m \xrightarrow{\sim} L_m$ (resp. $\tau: K \xrightarrow{\sim} L$ induced by $\tau_1: K_1 \xrightarrow{\sim} L_1$) in $(i)$ is uniquely determined by the condition in (i).
\end{thm}
Furthermore, in 2023, Corato and Sa\"idi proved an $m$-step solvable version of Theorem 1.1:
\begin{thm}[Theorem 6.1 in \cite{CS}]
Let $K,L$ be number fields, and let $m \geq 2$ be an integer. Let $\sigma_{m+3}: G_K^{m+3} \to G_L^{m+3}$ be a homomorphism of profinite groups. Consider the homomorphism $\sigma_m: G_K^m \to G_L^m$ induced by $\sigma_{m+3}$. Assume that $\sigma_m$ restricts to an open injection on every $(\star_{\ell})$-subgroup (c.f. Definition 1.20 in \cite{ST1}) of $G_K^{m}$ for every prime number $\ell$, and that $(G_L^{m+3})^{[m-1]} \subset \sigma_{m+3}(G_K^{m+3})$. \par 
Then there exists a unique field embedding $\tau:L_m \hookrightarrow K_m$ satisfying $\tau \sigma_m(g) = g \tau$ for all $g \in G_K^m$. Furthermore, $\tau$ restricts to a field embedding $L \hookrightarrow K$.
\end{thm}
In this paper, we prove the following $m$-step solvable version of Theorem 1.2.
\begin{thm}
Let $K$and $L$ be number fields. Let $m \geq 2$ be an integer, and let 
$$
\alpha_{m+3}: G_K^{m+3} \to G_L^{m+3}
$$
be an open homomorphism. We write 
$$
\alpha_m: G_K^m \to G_L^m
$$
for the open homomorphism induced by $\alpha_{m+3}$ (c.f. Proposition 2.2 in \cite{CS}). Then the followings are equivalent. \par 
(i) There exists a uniquely determined field embedding $\tilde \beta: L_m \hookrightarrow K_m$ satisfying $\tilde \beta \alpha_m(g) =  g \tilde\beta$ for all $g \in G_K^m$. \par 
(ii) The following diagram commutes
$$
\begin{tikzcd}
	{G_K^{m}} && {G_L^{m}} \\
	& {\widehat{\mathbb{Z}}^{\times}}
	\arrow[from=1-1, to=1-3]
	\arrow["{\chi_K}", from=1-1, to=2-2]
	\arrow["{\chi_L}"', from=1-3, to=2-2]
\end{tikzcd}
$$
where $\chi_K$ and $\chi_L$ are the cyclotomic characters.
\end{thm}
\section{The Proof of Uniqueness in Theorem 1.5}
In this section, we prove that in the implication (ii) $\implies$ (i) in Theorem 1.5, if such $\tilde \beta$ exists, then it is unique. Throughout this section we use the same notations as in Theorem 1.5 and assume that the assertion (ii) in Theorem 1.5 holds true. \par 
\begin{prop}
If a field embedding $\tilde \beta:L_m \hookrightarrow K_m$ as in Theorem 1.5 (i) exists then it is unique.
\end{prop}
\begin{proof}
Let $L_0/L$ be a finite Galois extension of $L$ contained in $L_m$. Then $\tilde \beta$ restricts to a field embedding $\tilde \beta|_{L_0}: L_0 \to K_0$ for some finite Galois extension $K_0/K$ contained in $K_m$. We have the following commutative diagram:
$$
\begin{tikzcd}
	{G_K^m} & {\text{Gal}(K_0/K)} \\
	{G_L^m} & {\text{Gal}(L_0/L)}
	\arrow[from=1-1, to=1-2]
	\arrow["{\alpha_m}", from=1-1, to=2-1]
	\arrow["{\eta}",from=1-2, to=2-2]
	\arrow[ from=2-1, to=2-2]
\end{tikzcd}
$$
where the map $\eta$ is induced by $\alpha_m$. Moreover, $\eta$ arises naturally from the following commutative diagram:
$$
\begin{tikzcd}
	{L_0} & {K_0} \\
	L & K
	\arrow["{\tilde \beta|_{L_0}}", from=1-1, to=1-2]
	\arrow[from=2-1, to=1-1]
	\arrow[from=2-1, to=2-2]
	\arrow[from=2-2, to=1-2]
\end{tikzcd}
$$
where all arrows are field embeddings. Write $\Gamma_{K_0}$ (resp. $\Gamma_{L_0}$) for the kernel of $G_K^m \twoheadrightarrow \text{Gal}(K_0/K)$ [resp. $\text{ker}(G_L^m \twoheadrightarrow \text{Gal}(L_0/L))$]. Moreover, we write $\widetilde{\Gamma}_{K_0}$ (resp. $\widetilde{\Gamma}_{L_0}$) for the inverse images of $\Gamma_{K_0}$ (resp. $\Gamma_{L_0}$) in $G_K^{m+3}$ (resp. in $G_L^{m+3}$). Thus, $\alpha_m$ determines an open homomorphism
$$
\gamma:G_{K_0}^3 \to G_{L_0}^3
$$
where $\widetilde{\Gamma}_{K_0}^3 \cong G_{K_0}^3$ (resp. $\widetilde{\Gamma}_{L_0}^3 \cong G_{L_0}^3$). \par 
Since $\gamma$ arises from $\tilde\beta|_{L_0}$, the induced map
$$
\gamma_1: G_{K_0}^{\text{ab}} \to G_{L_0}^{\text{ab}}
$$
also arises from $\tilde\beta|_{L_0}$. In particular, for a given prime number $\ell$, the induced map
$$
\tilde \gamma_1 : G_{K_0}^{\text{ab}} \otimes_{\widehat{\mathbb Z}} \mathbb Z/\ell \mathbb Z \to G_{L_0}^{\text{ab}} \otimes_{\widehat{\mathbb Z}} \mathbb Z/\ell \mathbb Z
$$
is Frobenius-preserving (c.f. Definition 2.7 in \cite{Ho3}). It follows from Corollary 2.8 in \cite{Ho3} that $\tilde \gamma_1$ determines a unique field embedding $\gamma_0:L_0 \hookrightarrow K_0$. Since $\tilde \gamma_1$ arises from $\tilde \beta|_{L_0}$, we can conclude that $\gamma_0 = \tilde \beta|_{L_0}$. \par 
Consider another field embedding $\tilde \beta': L_m \hookrightarrow K_m$. We have $\tilde \beta'|_{L_0}: L_0 \to K_0'$ for some finite Galois extension $K_0'$ of $K$ contained in $K_m$. In particular, $\tilde \beta'|_{L_0}: L_0 \to K_0K_0'$ and $\tilde \beta|_{L_0}: L_0 \to K_0K_0'$ coincide by repeating the above arguments. Thus, by composing with the natural inclusion $K_0K_0' \to K_m$, we can conclude that $\tilde \beta|_{L_0}: L_0 \to K_m$ and $\tilde \beta'|_{L_0}: L_0 \to K_m$ concide. Therefore,
$$
\tilde \beta' =\varprojlim_{L_0}~\tilde \beta'|_{L_0} =\varprojlim_{L_0} ~\tilde \beta|_{L_0} = \tilde \beta 
$$
i.e. $\tilde \beta$ is unique.
\end{proof}
\section{The Proof of Existence in Theorem 1.5}
We first introduce a few definitions. 
\begin{defn}[c.f. Definition 2.1 in \cite{Ho3}]
Let $K$ be a number field and let $\ell$ be a prime number. We write:
$$
A_{\ell}(K) := G_K^{\text{ab}} \otimes_{\widehat{\mathbb{Z}}} \mathbb{Z}/\ell\mathbb{Z}.
$$
\end{defn}
\begin{defn}[c.f. Definition 1.20 in \cite{ST1}]
Let $K$ be a number field and let $m \geq 2$ be an integer. Let $H \subset G_K^m$ be a closed subgroup and let $\ell$ be a prime number. We say that $H$ satisfies the condition $(\star_{\ell})$ if the followings hold. \par 
(i) $H$ fits into an exact sequence
$$
1 \to F_1 \to H \to F_2 \to 1
$$
where $F_1,F_2$ are free pro-$\ell$ groups of rank $1$, i.e. $F_1$ and $F_2$ are both isomorphic to $\mathbb Z_{\ell}$. \par 
(ii) $\mathcal{H}^2(H,\mathbb{F}_{\ell}) \neq 0$ [see Lemma 1.16 in \cite{ST1} for the definition of $\mathcal{H}^2(H,\mathbb{F}_{\ell})$].
\end{defn}
\begin{lem}[Compare with Lemma 3.1 in \cite{Ho3}]
Let $K$ be a number field and let $\mathfrak{p} \in \mathscr{P}_K^{\text{fin}}$. We write $p := \text{char}(\mathfrak{p})$ for the residue characteristic of $\mathfrak{p}$. Let $m \geq 1$ be an integer. We write $\chi: D_{\mathfrak{p}}^m \to \widehat{\mathbb{Z}}^{\times}$ for the cyclotomic character of $D_{\mathfrak{p}}^m$. Then the followings hold. \par 
(i) If $d_{\mathfrak{p}} = 1$, and $V \subset D_{\mathfrak{p}}^m$ is a pro-$2$ Sylow subgroup of the inertia subgroup of $D_{\mathfrak{p}}^m$, then $\chi(V) \neq 1$. \par 
(ii) Let $\ell$ be a prime number. Then $\ell = p$ if and only if $\chi(D_{\mathfrak{p}}^m)$ contains a closed subgroup isomorphic to $\mathbb{Z}_{\ell} \times \mathbb{Z}_{\ell}$. \par 
(iii) $f_{\mathfrak{p}} = 1$ if and only if the image of the composite
$$
D_{\mathfrak{p}}^m \xrightarrow{\chi} \widehat{\mathbb{Z}}^{\times} \twoheadrightarrow \prod_{\ell \neq p} \mathbb{Z}^{\times}_{\ell}
$$
contains the image of $p$ in $\prod_{\ell \neq p} \mathbb Z_{\ell}^{\times}$.\par 
(iv) If $p > 2$ and $f_{\mathfrak{p}} = 1$. Then the subset of $D_{\mathfrak{p}}^{\text{ab}} \otimes_{\widehat{\mathbb{Z}}} \mathbb{Z}/2\mathbb{Z}$ consisting of liftings of $\text{Frob}_{\mathfrak{p}}$ (i.e. the set $\text{FL}_{2}(\mathfrak{p})$, c.f. Definition 2.1 in \cite{Ho3}) coincides with the subset of $D_{\mathfrak{p}}^{\text{ab}} \otimes_{\widehat{\mathbb{Z}}} \mathbb{Z}/2\mathbb{Z}$ consisting  of the elements whose images under the map
$$
D_{\mathfrak{p}}^{\text{ab}} \otimes_{\widehat{\mathbb{Z}}} \mathbb{Z}/2\mathbb{Z} \to (\prod_{\ell \neq p}\mathbb{Z}_{\ell}^{\times}) \otimes_{\widehat{\mathbb{Z}}} \mathbb{Z}/2\mathbb{Z}
$$
are non-trivial. Here the map $D_{\mathfrak{p}}^m \to \prod_{\ell \neq p} \mathbb Z_{\ell}^{\times}$ is induced by the cyclotomic character.
\end{lem}
\begin{proof}
Since the cyclotomic character $\chi: D_{\mathfrak{p}} \to \widehat{\mathbb Z}^{\times}$ factors through $D_{\mathfrak{p}}^m$ for any $m \geq 1$, the proofs of assertions (i), (ii) and (iii) are induced by assertions (i),(ii) and (iii) in Lemma 3.1 in \cite{Ho3}. The proof of assertion (iv) is identical to the proof of asseriton (iv) in Lemma 3.1 in \cite{Ho3}.
\end{proof}
\begin{lem}
Let $m \geq 2$ be an integer and let $K$ be a number field. Then $G_K^m$ has no non-trivial $2$-torsion if and only if $K$ is totally imaginary.
\end{lem}
\begin{proof}
This is an immediate consequence of Proposition 1.9 in \cite{CS}.
\end{proof}
\begin{lem}[Compare with Lemma 3.3 in \cite{Ho3}]
Let $m \geq 1$ be an integer and let $K,L$ be number fields. Let $\chi_K,\chi_L$ be the cyclotomic characters associated to $G_K^m$ and $G_L^m$, respectively. Let $\alpha_{m+2}: G_K^{m+2} \to G_L^{m+2}$ be an open homomorphism and $\alpha_m: G_K^m \to G_L^m$ the open homomorphism induced by $\alpha_{m+2}$ (c.f. Proposition 2.2 in \cite{CS}). Suppose that the following conditions hold. \par 
(1) The following diagram commutes
$$
\begin{tikzcd}
	{G_K^m} && {G_L^m} \\
	& {\widehat{\mathbb{Z}}^{\times}.}
	\arrow["{\alpha_m}", from=1-1, to=1-3]
	\arrow["{\chi_K}", from=1-1, to=2-2]
	\arrow["{\chi_L}"', from=1-3, to=2-2]
\end{tikzcd}
$$
(2) $L$ is totally imaginary (i.e. $L$ has no real embeddings). \par 
Then the followings hold. \par 
(i) Let $\mathfrak{p} \in \mathscr{P}_K^{\text{fin},f=1}$ be such that $p_{\mathfrak{p}} > 2$ and $d_{\mathfrak{p}} = 1$. Then there exists a unique $\mathfrak{q} \in \mathscr{P}_L^{\text{fin}}$ such that the following conditions are satisfied: \par 
(a) The image of $D_{\mathfrak{p}}^m \subset G_K^m$ under $\alpha_m$ is contained in $D_{\mathfrak{q}}^m$ ($D_{\mathfrak{p}}^m$ and $D_{\mathfrak{q}}^m$ are decomposition subgroups at $\mathfrak{p}$ and $\mathfrak{q}$, respectively, and are well-defined up to conjugation). \par 
(b) $\mathfrak{p}$ and $\mathfrak{q}$ have the same residue characteristic (i.e. $p_{\mathfrak{p}} = p_{\mathfrak{q}}$). \par
(c) $f_{\mathfrak{q}} = 1$. \par 
(ii) The map $\alpha_{m}$ induces a continuous map of $\widehat{\mathbb{Z}}$-modules
$$
\alpha^{\text{ab}}: G_K^{\text{ab}} \otimes_{\widehat{\mathbb{Z}}} \mathbb{Z}/2\mathbb{Z} \to G_L^{\text{ab}} \otimes_{\widehat{\mathbb{Z}}} \mathbb{Z}/2\mathbb{Z}
$$
which is Frobenius-preserving (c.f Definition 2.7 in \cite{Ho3}).
\end{lem}
\begin{proof}
First, we prove (i). Let $V \subset D_{\mathfrak{p}}^{m+1}$ be a pro-$2$ Sylow subgroup. Write $V^* = V \cap I_{\mathfrak{p}}^{m+1}$ where $I_{\mathfrak{p}}^{m+1}$ denotes the inertia subgroup of $D_{\mathfrak{p}}^{m+1}$. Since $2$ is not the residue characteristic of $\mathfrak{p}$, we may conclude that $V$ fits into the exact sequence
$$
1 \to V^* \to V \to \mathbb{Z}_2 \to 1
$$
where $V^* \xrightarrow{\sim} \mathbb{Z}_2$ (c.f. Proposition 1.1 (vi) in \cite{ST1}). Since $L$ is totally imaginary, it holds by Lemma 3.4 that $G_L^{m+1}$ has no $2$-torsions. In particular, by the discussions in p.11 in \cite{CS}, and Lemma 3.3 (i), we may conclude that $\alpha_{m+1}$ restricts to an injection on $V$. Then by Proposition 2.5 in \cite{CS}, the image of $V$ is a $(\star_2)$ subgroup of $G_L^{m+1}$. We write $V_0$ for the image of $V$ in $G_L^{m+1}$. Now by the local theory established in \cite{ST1} (more precisely, Proposition 1.22 in \cite{ST1}) $V_0$ determines a unique prime $\mathfrak{q} \in \mathscr{P}_L^{\text{fin}}$, more specifically, a unique decomposition subgroup $D_{\mathfrak{q}}^m \subset G_L^{m}$. In particular, by Proposition 3.10 in \cite{CS}, the image of $D_{\mathfrak{p}}^m$ is contained in $D_{\mathfrak{q}}^m$. Now by condition (1), the following diagram commutes
$$
\begin{tikzcd}
	{D_{\mathfrak{p}}^m} && {D_{\mathfrak{q}}^m} \\
	& {\widehat{\mathbb{Z}}^{\times}.}
	\arrow[from=1-1, to=1-3]
	\arrow["{\chi_K|_{D_{\mathfrak{p}}^m}}"', from=1-1, to=2-2]
	\arrow["{\chi_L|_{D_{\mathfrak{q}}^m}}", from=1-3, to=2-2]
\end{tikzcd}
$$
It follows from Lemma 3.3 (ii), that $p_{\mathfrak{p}} = p_{\mathfrak{q}}$. Then it follows from Lemma 3.3 (iii) that we have $f_{\mathfrak{q}} = 1$. This completes the proof of assertion (i). The proof of assertion (ii) is identical to the proof of Lemma 3.3 (ii) in \cite{Ho3}.
\end{proof}
Now we can prove Theorem 1.5. 
\begin{proof}[Proof of Theorem 1.5]
Notice that the proof of (1) $\implies$ (2) is immediate. We only need to verify (2) $\implies$ (1). \par 
Consider the open homomorphism
$$
\alpha_m: G_K^m \to G_L^m
$$
induced by $\alpha_{m+3}$. Consider an open normal subgroup $H' \subset G_L^m$ corresponding to a totally imaginary finite extension $F'/L$ contained in $L_m$. We write $H \subset G_K^m$ for the inverse image of $H'$ by $\alpha_m$, and $F/K$ for the finite extension of $K$ corresponding to $H$ which is contained in $K_m$. \par 
Write $\widetilde{H} \subset G_K^{m+3}$ (resp. $\widetilde{H'} \subset G_L^{m+3}$) for the inverse image of $H$ (resp. $H'$) via the natural surjection $G_K^{m+3} \twoheadrightarrow G_K^m$ (resp. $G_L^{m+3} \twoheadrightarrow G_L^m$). Thus, we have the following isomorphisms
$$
\widetilde{H}^3 \cong G_F^3 ~;~ \widetilde{H'}^3 \cong G_{F'}^3.
$$
Now consider the following commutative diagram with exact rows:
$$
\begin{tikzcd}
1 \arrow[r] & \widetilde{H} \arrow[r] \arrow[d,twoheadrightarrow] & G_K^{m+3} \arrow[d,"\alpha_{m+3}"] \arrow[r] & \text{Gal}(F/K) \arrow[d,"\beta"] \arrow[r] & 1 \\
1 \arrow[r] & \widetilde{H'} \arrow{r} & G_L^{m+3} \arrow[r] & \text{Gal}(F'/L) \arrow[r] & 1.
\end{tikzcd}
$$
We have the following commutative diagram:
$$
\begin{tikzcd}
G_F^{\text{ab}} \arrow[r,twoheadrightarrow] \arrow[d,twoheadrightarrow] & A_2(F) \arrow[d,"\varphi",twoheadrightarrow] \\
G_{F'}^{\text{ab}} \arrow[r,twoheadrightarrow] & A_2(F')
\end{tikzcd}
$$
where $A_2(F) := G_F^{\text{ab}} \otimes_{\widehat{\mathbb{Z}}} \mathbb{Z}/2\mathbb{Z}$ (resp. $A_2(F') := G_{F'}^{\text{ab}} \otimes_{\widehat{\mathbb{Z}}} \mathbb{Z}/2\mathbb{Z}$), the horizontal arrows are natural surjections, and the vertical arrows are surjections induced by the surjection $\widetilde{H} \twoheadrightarrow \widetilde{H'}$ and are equivariant with respective to $\beta$.\par 
On the other hand, since the diagram
$$
\begin{tikzcd}
	{G_K^{m+3}} && {G_L^{m+3}} \\
	& {\widehat{\mathbb{Z}}^{\times}}
	\arrow["{\alpha_{m+3}}", from=1-1, to=1-3]
	\arrow["{\chi_K}", from=1-1, to=2-2]
	\arrow["{\chi_L}"', from=1-3, to=2-2]
\end{tikzcd}
$$
commutes, we may conclude that we have a commutative diagram:
$$
\begin{tikzcd}
	{G_F^{\text{ab}}} && {G_{F'}^{\text{ab}}} \\
	& {\widehat{\mathbb{Z}}^{\times}}
	\arrow[from=1-1, to=1-3]
	\arrow["{\chi_F}", from=1-1, to=2-2]
	\arrow["{\chi_{F'}}"', from=1-3, to=2-2]
\end{tikzcd}
$$
where $\chi_F$ and $\chi_{F'}$ are the cyclotomic characters of $F$ and $F'$, respectively. By Lemma 3.5 it holds that the map $ \varphi:A_2(F) \twoheadrightarrow A_2(F')$ is Frobenius-preserving (c.f. Definition 2.7 in \cite{Ho3}). We conclude by Corollary 2.8 in \cite{Ho3} that the map $A_2(F) \twoheadrightarrow A_2(F')$ determines a unique field embedding $\phi: F' \hookrightarrow F$. \par 
Now take $g \in \text{Gal}(F/K)$ and $g' := \beta(g) \in \text{Gal}(F'/L)$. We have the following commutative diagram:
$$
\begin{tikzcd}
	{F'} & F \\
	{F'} & F
	\arrow["\phi", hook, from=1-1, to=1-2]
	\arrow["{g'}"', from=1-1, to=2-1]
	\arrow["g", from=1-2, to=2-2]
	\arrow["{\phi'}", hook, from=2-1, to=2-2]
\end{tikzcd}
$$
where $\phi'$ is defined to be $\phi' := g \phi g'^{-1}$.\par
Since the map $\varphi: A_2(F) \twoheadrightarrow A_2(F')$ is equivariant with respect to $\beta$, we have $\varphi(g.x) = g'\varphi(x)$ for all $x \in A_2(F)$. By Corollary 2.8 in \cite{Ho3}, $\varphi$ arises uniquely from $\phi$. On the other hand, $\phi'$ also determines a map $\varphi' : A_2(F) \twoheadrightarrow A_2(F')$ such that the following diagram commutes:
$$
\begin{tikzcd}
A_2(F) \arrow[r,"\varphi",twoheadrightarrow] \arrow[d,"\wr"] & A_2(F') \arrow[d,"\wr"] \\
A_2(F) \arrow[r,"\varphi'",twoheadrightarrow] & A_2(F')
\end{tikzcd}
$$ 
where the left vertical arrow is given by the action of $g$ and the right vertical arrow is given by the action of $g'$. In particular, $\varphi = \varphi'$ by the commutativity of the above diagram and the fact that $\varphi$ is $\beta$-equivariant. Hence $\phi' = \phi$. So $\phi = g \phi g'^{-1} = g \phi \beta(g)^{-1} \implies \phi \beta(g) = g \phi$, i.e. $\beta$ arises uniquely from $\phi$.\par
On the other hand, we have the following identifications:
$$
G_K^m \cong \varprojlim_F \text{Gal}(F/K)~;~ G_L^m \cong \varprojlim_{F'} \text{Gal}(F'/L)~~~~~~(\dagger)
$$
where $F'$ ranges over all finite Galois totally imaginary extensions of $L$ contained in $L_m$, and $F$ ranges over finite Galois extensions of $K$ contained in $K_m$ corresponding to $F'$. \par 
Write $\text{Hom}_{\beta}(F',F)$ for the set of field embeddings $F' \to F$ that arise from $\beta$. It follows from the identification $(\dagger)$ that we have
$$
\varprojlim_{F'} ~\text{Hom}_{\beta}(F',F) \xrightarrow{\sim} \text{Hom}_{\alpha_m}(L_m,K_m)
$$
where $\text{Hom}_{\alpha_m}(L_m,K_m)$ is the set of all field embeddings $L_m \to K_m$ arising from $\alpha_m$. Furthermore, it follows from the various constructions involved, that $\text{Hom}_\beta(F',F)$ is (finite) non-empty, hence $\text{Hom}_{\alpha_m}(L_m,K_m)$ is also non-empty since it is the inverse limit of non-empty finite sets. Thus, we can conclude that there exists a field embedding $\tilde \beta :L_m \to K_m$ arising from $\alpha_m$, and by Proposition 2.1 $\tilde \beta$ is unique.
\end{proof}
\bibliography{ref.bib}

@unpublished{CS,
author = {Corato, Alberto and Sa{\"i}di, Mohamed},
title = {The $m$-step Solvable Hom-Form of Birational Anabelian Geometry for Number Fields},
year = {2023},
note = {Preprint, https://arxiv.org/abs/2307.07834},
}

@article{Ho3,
author = {Yuichiro Hoshi},
title = {{Homomorphisms of global solvably closed Galois groups compatible with cyclotomic characters}},
volume = {77},
journal = {Tohoku Mathematical Journal},
number = {1},
publisher = {Tohoku University, Mathematical Institute},
pages = {17 -- 32},
year = {2025},
doi = {10.2748/tmj.20230414},
URL = {https://doi.org/10.2748/tmj.20230414}
}

@article{ST1,
url = {https://doi.org/10.1515/crelle-2022-0025},
title = {The m-step Solvable Anabelian Geometry of Number Fields},
author = {Sa{\"i}di, Mohamed and Tamagawa, Akio},
pages = {~153--186},
volume = {2022},
number = {789},
journal = {Journal für die reine und angewandte Mathematik (Crelle's Journal)},
year = {2022}
}

@article{Uchida1,
author = {Uchida, K{\^o}ji},
title = {{Isomorphisms of Galois groups of solvably closed Galois extensions}},
volume = {31},
journal = {Tohoku Mathematical Journal},
number = {3},
publisher = {Tohoku University, Mathematical Institute},
pages = {~359 -- 362},
year = {1979},
URL = {https://doi.org/10.2748/tmj/1178229803},
}

@article{Uchida3,
author = {Uchida, K{\^o}ji},
title = {{Homomorphisms of Galois groups of solvably closed Galois extensions}},
volume = {33},
journal = {Journal of the Mathematical Society of Japan},
number = {4},
publisher = {Mathematical Society of Japan},
pages = {~595 -- 604},
year = {1981},
doi = {10.2969/jmsj/03340595},
URL = {https://doi.org/10.2969/jmsj/03340595},
}
\bibliographystyle{alpha}

\end{document}